\documentclass[a4paper]{article}

\usepackage{amsmath}
\usepackage{amssymb}
\usepackage{geometry}
\usepackage{amsthm}
\usepackage{graphicx}
\usepackage[all]{xy}
\usepackage{multirow} 
\usepackage[usenames]{color}
\usepackage{url}
\usepackage[colorlinks=true,urlcolor=blue,citecolor=black,linkcolor=darkgray,draft=false]{hyperref}
\usepackage[textwidth=2cm,color=green!40,textsize=scriptsize]{todonotes}

\newtheorem{theorem}{Theorem}
\newtheorem{lemma}[theorem]{Lemma}
\newtheorem{proposition}[theorem]{Proposition}
\newtheorem{corollary}[theorem]{Corollary}
\theoremstyle{remark}
\newtheorem{remark}{Remark}
\newtheorem{example}[remark]{Example}
\theoremstyle{definition}

\newcommand{\trace}{\mathrm{trace}}
\newcommand{\rank}{\mathrm{rank}}
\newcommand{\range}{\operatorname{Im}}
\newcommand{\nullspace}{\operatorname{Ker}}

\newcommand{\R}{\mathbb{R}}

\newcommand{\lb}{\overline{\lambda}}
\newcommand{\zb}{\overline{z}}

\renewcommand{\vec}[1]{\boldsymbol{#1}}

\newcommand{\vectwo}[2]{ \left[\!\!
			     \begin{array}{c}
                             #1\\
			     #2 
                             \end{array} \!\! \right]}

\newcommand{\twotwomat}[4]{\left( \begin{array}{cc}
                                   #1 & #2\\#3 & #4
                                  \end{array}
\right) }
\newcommand{\mub}{\overline{\mu}}
\newcommand{\Yb}{\overline{Y}}
\newcommand{\lambdamin}{\lambda_{\textrm{min}}}

\title{A Class of Semidefinite Programs with rank-one solutions}
\author{Guillaume Sagnol \\
       {\small Inria Saclay -- \^Ile-de-France \&  Centre de Math\'ematiques Appliqu\'ees (CMAP),}\\
       {\small \'Ecole Polytechnique, France.} \\
       {\small guillaume.sagnol@inria.fr}
}

\begin{document}

\maketitle

\begin{abstract}
We show that a class of semidefinite programs (SDP) admits a solution that is a positive semidefinite
matrix of rank at most $r$, where $r$ is the rank of the matrix involved in the objective function of the SDP.
The optimization problems of this class are semidefinite packing problems, 
which are the SDP analogs to vector packing problems.
Of particular interest is the case in which our result guarantees the existence of a solution
of rank one: we show that the computation of this solution actually reduces to a 
Second Order Cone Program (SOCP). 
We point out an application in statistics, in the optimal design of experiments. 
\end{abstract}
\paragraph{Keywords}
SDP, Semidefinite Packing Problem, rank~$1$-solution, Low-rank solutions, SOCP, Optimal Experimental Design, Multiresponse experiments.

\section{Introduction}
In this paper, we study \emph{semidefinite packing problems}. The latter, which are the semidefinite programming (SDP) analogs to the packing problems in linear programming, can be written as:
\begin{align} \tag{P}
 \max &\quad \langle C,X \rangle \label{SDPPacking}\\
 \textrm{s.t.} &\quad \langle M_i,X \rangle \leq b_i, \qquad i \in [l],\nonumber\\
 & \quad X \succeq 0, \nonumber
\end{align}
where $C \succeq 0$, and $M_i \succeq 0,\ i\in[l]$. The notation $X\succeq 0$ indicates that $X$ belongs to the
set $\mathbb{S}_n^+$ of $n \times n$ symmetric positive semidefinite matrices. Similarly, $X\succ0$ stands for $X\in\mathbb{S}_n^{++}$,
the set of $n \times n$ symmetric positive definite matrices. 
The space of $n \times n$ symmetric matrices $\mathbb{S}_n$ is equipped with the inner product $\langle A,B \rangle = \mathrm{trace}(A^TB)$. We also make
use of the standard notation $[l]:=\{1,\ldots,l\}$, and we use boldface letters to denote vectors. We denote
the nullspace (resp.\ the range) of a matrix~$A$ by $\nullspace\ A$ (resp.\ $\range\ A$).

Semidefinite packing problems were introduced by Iyengar, Phillips and Stein~\cite{IPS05}. They showed
that these arise in many applications such as relaxations of combinatorial optimization problems or maximum variance unfolding, and gave an algorithm
to compute approximate solutions, which is faster than the commonly used interior point methods.

Our main result is that when the matrix $C$ is of rank $r$, Problem~\eqref{SDPPacking} has a solution
that is of rank at most $r$ (Theorem~\ref{theo:SDPSol1}). In particular, when $r=1$, the optimal SDP variable $X$ can
be factorized as $\vec{x}\vec{x^{T}}$, and we show that finding $\vec{x}$
reduces to a Second-Order Cone Program (SOCP)
which is computationally more tractable than the initial SDP. 
We present this result and some applications in Section~2. Then, we
extend our result to a wider class of semidefinite programs (Theorems~\ref{theo:SDPSol1Asy}
and~\ref{theo:SDPSol1Ext}),
in which not all the constraints are of packing type. The proofs of the results of
Section~\ref{sec:mainresult} are given in Section~\ref{sec:proof}. Theorems~\ref{theo:SDPSol1Asy}
and~\ref{theo:SDPSol1Ext} are proved in appendix.

\paragraph{Related work}

Solutions of small rank of semidefinite programs
have been extensively studied over the past years.
Barvinok~\cite{Bar95} and Pataki~\cite{Pat98} discovered independently that any SDP with $l$ constraints
has a solution $X^*$ whose rank is at most
$$r^* = \left\lfloor \frac{\sqrt{8l+1}-1}{2} \right\rfloor,$$
where $\lfloor \cdot \rfloor$ denotes the integer part.
This was one of the motivations of Burer and Monteiro for developing the SDPLR solver~\cite{BM03},
which searches a solution of the SDP in the form $X=RR^T$, where $R$ is a $n \times r^*$ matrix.
The resulting problem is non-convex, and so the augmented Lagrangian algorithm proposed in~\cite{BM03} is not guaranteed to
converge to a global optimum. However, it performs remarkably well in practice, and some conditions
which ensure that the returned solution is an optimum of the SDP are provided in~\cite{BM05}.
Our result shows that for a semidefinite packing problem in which the matrix $C$ has rank $r$,
one can force the matrix $R$ to be of size $n \times r$ (rather than  $n \times r^*$),
which can lead to considerable gains in computation time when $r$ is small.
\par
\vspace{11pt}

We point out that the ratio between the optimal value of Problem~\eqref{SDPPacking} and the value of its best solution of rank one 
has been studied by Nemirovski, Roos, and Terlaky~\cite{NRT99}. They show
that the value $v^*$ of the SDP and the value $v_1^*$ of its best rank-one solution satisfy:
\begin{equation}
 v^* \geq v_1^* \geq \frac{1}{2 \ln(2 l\mu)} v^*, \qquad \textrm{where}\ \mu=\min(l,\max_{i\in[l]}\ \rank\ {M_i}). \label{gap}
\end{equation}
This ratio can be considerably reduced in particular configurations, but to
the best of our knowledge, the fact that the gap in~\eqref{gap} vanishes when the matrix $C$ in the objective function is of rank $1$
is new, except in the particular case in which every $M_i$ is of rank $1$, too~\cite{Rich08}.

\section{Main result and consequences} 

In this section, we state the main result of this article and point out an application to statistics.
We also discuss the significance of our result for combinatorial optimization problems (the
hypothesis on the rank of the matrix $C$ appears to be very restrictive).
The results of this section are proved in Section~\ref{sec:proof}.

\subsection{The main result} \label{sec:mainresult}
We start with an algebraic characterization of the semidefinite packing problems that
are feasible and bounded.

\begin{theorem}
\label{theo:feasibility-Boundness}
Problem~\eqref{SDPPacking} is feasible if and only if every $b_i$ is nonnegative. Moreover if 
Problem~\eqref{SDPPacking}
is feasible, then this problem is bounded if and only if the range of $C$ is included in the range of
$\sum_i M_i$.
\end{theorem}

The reader should note that the range inclusion condition in Theorem~\ref{theo:feasibility-Boundness}
is is fact equivalent to the feasibility of the Lagrangian dual of Problem~\eqref{SDPPacking}:
 \begin{align} \tag{D}
 \min_{\vec{\mu}\geq0} & \quad \vec{\mu^T b} \label{SDPPackingDual}\\
 \textrm{s.t.} &\quad \sum_i \mu_i M_i \succeq C. \nonumber
\end{align}

The main result of this article follows:

\begin{theorem}
\label{theo:SDPSol1}
We assume that the conditions of Theorem~\ref{theo:feasibility-Boundness} are fulfilled, so that
Problem~\eqref{SDPPacking} is feasible and bounded. If $\operatorname{rank} C=r$,
then the semidefinite packing problem~\eqref{SDPPacking}
has a solution which is a matrix of rank at most $r$. 
\end{theorem}

A consequence of Theorem~\eqref{theo:SDPSol1} is that when the matrix in the objective function
is of rank $1$ $(C=\vec{cc}^T)$, the computation of a solution $X$ of Problem~\eqref{SDPPacking} reduces to the
computation of a vector $\vec{x}$ such that $X=\vec{xx^{T}}$. The next result shows that this can be done very efficiently by a Second Order Cone Program (SOCP).

\begin{corollary} \label{coro:SOCP}
We assume that the conditions of Theorem~\ref{theo:feasibility-Boundness} are fulfilled, and that 
$C=\vec{c}\vec{c^T}$ for a vector $\vec{c} \in \mathbb{R}^n$ (i.e.\ $\rank\ C=1$).
Then, Problem~\eqref{SDPPacking} reduces to the SOCP:
 \begin{align}
 \max_{x\in \R^n} &\quad \vec{c^T x} \label{SOCPPacking}\\
 \textrm{s.t.} &\quad \Vert A_i \vec{x} \Vert_2 \leq \sqrt{b_i}, \qquad i=1\in[l],\nonumber
\end{align}
where the matrices $A_i$ are such that $M_i=A_i^T A_i$. Moreover, if $\vec{x}$ is any optimal solution of Problem~\eqref{SOCPPacking}, then
$X=\vec{xx^{T}}$ is an optimal solution of Problem~\eqref{SDPPacking}, and the optimal value
of~\eqref{SDPPacking} is $(\vec{c^T x})^2$.
\end{corollary}

\begin{proof}
The SOCP~\eqref{SOCPPacking} is simply obtained from~\eqref{SDPPacking} by substituting $\vec{xx^T}$ from $X$
and $A_i^TA_i$ from $M_i$. The objective function $\langle C,X \rangle$ becomes $(\vec{c^T x})^2$, and we can
remove the square by noticing that $\vec{c^T x}\geq0$ without loss of generality, since if $\vec{x}$ is
optimal, so is $-\vec{x}.$
\end{proof}

In fact, the proof of Theorem~\ref{theo:SDPSol1} relies on the projection of Problem~\eqref{SDPPacking}
on an appropriate subspace, which lets the reduced semidefinite packing problem be strictly feasible, as well as
its dual. This reduction is not only of theoretical interest, since in some cases it
may yield some important computational savings. Therefore, we next state this result as a proposition.

Let $\mathcal{I}_0:=\{i\in[l]: b_i=0\}$ and
$\mathcal{I}:=[l] \setminus \mathcal{I}_0$. Let the columns of the $n\times n_0$ matrix $U$ form an
orthonormal basis of $\range(\sum_{i\in[l]} M_i)$, and the columns of the $n_0 \times n'$ matrix $V$
form an orthonormal basis of $\nullspace(U^T \sum_{i\in \mathcal{I}_0} M_i U).$
We further define
$C':=(UV)^TC(UV)\in\mathbb{S}_{n'}^+$ and $M_i':=(UV)^TM_i(UV)\in\mathbb{S}_{n'}^+$ (for $i\in\mathcal{I}$),
and we consider the 
reduced problem
\begin{align}\tag{P'} \label{Pprime}
 \max_{Z\in\mathbb{S}_{n'}^+} &\quad \langle C',Z \rangle \\
 \textrm{s.t.} &\quad \langle M_i',Z \rangle \leq b_i, \qquad i \in \mathcal{I}.\nonumber
\end{align}

\begin{proposition} \label{prop:reduction}
We assume that the conditions of Theorem~\ref{theo:feasibility-Boundness} are fulfilled, so that
Problem~\eqref{SDPPacking} is feasible and bounded. Then, the following properties hold:
\begin{itemize}
 \item[(i)]  Problem~\eqref{Pprime} is strictly feasible, i.e.\
$\exists \overline{Z}\succ0:\ \forall i\in\mathcal{I}, \langle M_i', \overline{Z} \rangle < b_i$;
 \item[(ii)] The Lagrangian dual of~\eqref{Pprime} is strictly feasible, i.e.
$\exists \vec{\mub}>\vec{0}:\ \sum_{i\in \mathcal{I}} \mub_i M_i' \succ C'$;
 \item[(iii)] If $Z$ is a solution of Problem~\eqref{Pprime},
then $X:=(UV)Z(UV)^T$ is an optimal solution of Problem\eqref{SDPPacking} (which of course satisfies
$\rank\ X\leq \rank\ Z$ and $\langle C,X\rangle=\langle C',Z\rangle$).
\end{itemize}
\end{proposition}

The present work grew out from an application to networks~\cite{BGSagnol08Rio},
in which the traffic between any two pairs of nodes must
be inferred from a set of measurements. This
can be modeled by the theory of optimal experimental design, which leads to a large
SDP. Standard solvers
relying on interior points methods, like SeDuMi~\cite{sedumi}, cannot handle problems of this size.
However, in a followup work relying on the present reduction
to an SOCP~\cite{SagnolGB10ITC}, we solve within seconds the same instances in SeDuMi.
We next present this application. 

\subsection{Application to the optimal design of experiments}
An interesting application arises in statistics,
in the design of optimal experiments (for more
details on the subject, the reader is referred to Pukelsheim~\cite{Puk93}).
An experimenter wishes to estimate the quantity $\vec{c^T} \vec{\theta}$, 
where $\vec{\theta}$ is an unknown $n-$dimensional parameter, and $\vec{c}$ is a vector of $n$ coefficients.
To this end, she disposes of $l$ available experiments, each one giving a linear
measurement of the parameter $\vec{y_i}=A_i\vec{\theta}$, up to a (centered) measurement noise.
If the amount of experimental effort spent on the $i^\textrm{th}$ experiment is $w_i$, it is known
that the variance of the best linear unbiased estimator for $\vec{c^T} \vec{\theta}$ is $\vec{c^T} (\sum_i  w_i M_i)^\dagger \vec{c}$,
where $M_i=A_i^T A_i$, and $M^\dagger$ denotes the Moore-Penrose inverse of $M$. The problem of
distributing the experimental effort so as to minimize this variance is called the ``$\vec{c}-$optimal problem'', and can be formulated as:
\begin{align}
 \min_{\vec{w}\geq\vec{0}}\quad& \vec{c^T} (\sum_i  w_i M_i)^\dagger \vec{c}  \label{cOpt}\\
\textrm{s.t.}\quad&\sum_{i=1}^l w_i=1. \nonumber
\end{align}
It is classical to reformulate this problem as
a semidefinite program, by using the Schur complement lemma and duality theory (see e.g.~\cite{Rich08,Sagnol09SOCP}). The
$c-$optimal SDP already appeared in Pukelsheim and Titterington~\cite{Puk80}, hidden under a more general form:
\begin{align}
 \max &\quad \vec{c^T} X \vec{c} \label{SDPcPacking}\\
 \textrm{s.t.} &\quad \langle M_i,X \rangle \leq 1, \qquad i\in[l],\nonumber\\
 & \quad X \succeq 0. \nonumber
\end{align}
In this problem, the design variable $\vec{w}$ is proportional to the dual variable associated to the
constraints $\langle M_i,X \rangle~\leq~1$.
Note that this is a semidefinite packing problem, in which the matrix defining the objective function has rank $1$
($C=\vec{c} \vec{c^T}$). More generally, if we want to
estimate simultaneously $r$ linear functions of the parameter
$\vec{\zeta}=(\vec{c_1^T} \vec{\theta},\ldots, \vec{c_r^T} \vec{\theta})$, 
the best unbiased estimator $\hat{\zeta}$ is now an $r-$dimensional vector with covariance matrix
$$\operatorname{Cov}_{\vec{w}}(\vec{\hat{\zeta}}):=K^T  (\sum_{k=1}^l  w_k M_k)^\dagger K,$$
where $K=[\vec{c_1},\ldots,\vec{c_r}]$.
Several criteria can be used for this experimental design problem. Popular ones
are the $A-$criterion and the $E-$criterion, which aim at minimizing respectively the trace and the largest
eigenvalue of $\operatorname{Cov}_{\vec{w}}(\vec{\hat{\zeta}})$.
These optimization problems can also be formulated as semidefinite packing problems. For $A-$optimality,
this \emph{packing} formulation is given in~\cite{Sagnol09SOCP}:
\begin{align}
 \max &\quad \vec{\tilde{c}^T} X \vec{\tilde{c}} \label{Aopt}\\
 \textrm{s.t.} &\quad \langle \tilde{M_i},X \rangle \leq 1, \qquad i\in[l],\nonumber\\
 & \quad X \succeq 0, \nonumber
\end{align}
where $\vec{\tilde{c}}=[\vec{c_1^T},\ldots,\vec{c_r^T}]^T$, and $\tilde{M_i}$ is a block-diagonal matrix which contains $r$ times
the block $M_i$ on its main diagonal. The matrix in the objective function is of rank $1$ ($C=\vec{\tilde{c}\tilde{c}^T}$), and
so Problem~\eqref{Aopt} reduces to a SOCP by Corollary~\ref{coro:SOCP}.
This reduction is of great interest for the computation of optimal experimental designs,
because SOCP solvers are much more efficient than SDP solvers, and take
advantage of the sparsity of the matrices $A_i$ (whereas the matrices $M_i=A_i^TA_i$ used in the original SDP
formulation~\eqref{Aopt} are \textit{not very sparse} in general).

The $E-$optimal design SDP is presented in~\cite{VBW98}
(for the special case in which $C=I$), and takes exactly the form of the semidefinite packing problem~\eqref{SDPPacking},
with $b_i=1$ for all $i\in[l]$ and $C=KK^T=\sum_{i=1}^r \vec{c_i} \vec{c_i}^T$.
Here, the matrix $C$ has rank $r$, and so Theorem~\ref{theo:SDPSol1} indicates that the
$E-$optimal design SDP has a solution which is a matrix of rank at most $r$.
This suggests the use of specialized low rank solvers for this SDP when $r$ is small (cf.\ the paragraph
``\emph{Related work}''  at the end of the introduction),
which can lead to a considerable improvement in terms of computation time.

\subsection{Relation with combinatorial optimization}
SDP relaxations of combinatorial optimization problems have motivated the authors
of~\cite{IPS05} to study semidefinite packing problems. Hence, we discuss the significance of
our result for this class of problems in this section. 

Semidefinite programs have been used extensively to formulate relaxations of NP-hard combinatorial optimization problems after
the work of Goemans and Williamson on the approximability of MAXCUT~\cite{GW95}. These SDP relaxations often lead to
optimal solutions of the related combinatorial optimization problems whenever the solution of the SDP is of small rank.
As shown by
Iyengar et.\ al.~\cite{IPS05}, SDP relaxations of many combinatorial optimization problems can be cast as semidefinite packing programs.
Our result therefore identifies a subclass of 
combinatorial optimization problems which are solvable in polynomial time.
Unfortunately, this promising statement only helped us to identify trivial instances so far.
For example, the MAXCUT semidefinite packing problem~\cite{IPS05}
yields an exact solution of the combinatorial problem whenever it has a rank~$1$ solution.
The matrix $C$ in the objective function of this SDP is the Laplacian of the graph, and so it is known
that $$\rank\ C=N-\kappa,$$
where $N$ is the number of vertices and $\kappa$ is the number of connected components in the graph.
Our result therefore states that if a graph of $N$ vertices has $N-1$ connected components,
then it defines a MAXCUT instance that is solvable in polynomial time.  Such graphs actually consist
in a pair of connected vertices, plus $N-2$ isolated vertices, and the related MAXCUT instance is trivial.

Another limitation for the application of our theorem in this field is that most semidefinite packing
problems arising in
combinatorial optimization (including but not limited to the Lov\'asz $\vartheta$ function SDP~\cite{Lov79} and
the related Szegedy number SDP~\cite{Szeg94},
the vector colouring SDP~\cite{KMS98}, the sparsest cut SDP~\cite{ARV09}
and the sparse principal components analysis SDP~\cite{dEGJL07})
can be written in the form of~\eqref{SDPPacking}, with an additional trace equality constraint $\trace(X)=1$.
In fact, we can show that if such an ``equality constrained'' problem
is strictly feasible, then it is equivalent to the following ``classical'' semidefinite packing problem:

\begin{align}
 \max &\quad \langle C+\lambda I,X \rangle \label{SDPPackingTrace} - \lambda \\
 \textrm{s.t.} &\quad \langle M_i,X \rangle \leq b_i, \qquad i \in [l],\nonumber\\
 & \quad \trace X \leq 1,\nonumber \\
 & \quad X \succeq 0, \nonumber
\end{align}
where $\lambda$ is any scalar larger than $|\lambda^*|$, where  $\lambda^*$ is the optimal Lagrange multiplier associated to the constraint $\trace(X)= 1$ (we omit the proof
of this statement which is of secondary importance in this article). Since $C+\lambda I$ is a full rank matrix, our result does not seem to yield
any valuable information for this class of problems.

\section{Extension to ``combined'' problems}
The proof of our main result also applies to a wider class of semidefinite programs, which can be written as:

\begin{align} \tag{$P_{\scriptscriptstyle{\textrm{CMB}}}$}
 \sup_{X,Y,\vec{\lambda}} &\quad \langle C,X \rangle + \langle R_0,Y \rangle +\vec{h_0}^T\vec{\lambda}\label{SDPPackingExt}\\
 \textrm{s.t.} &\quad \langle M_i,X \rangle \leq b_i+\langle R_i,Y \rangle + \vec{h_i}^T\vec{\lambda}, \qquad i \in [l],\nonumber\\
 & \quad X \in \mathbb{S}_n^+,\ Y\in\mathbb{S}_p^+,\ \vec{\lambda}\in\R^q, \nonumber
\end{align}
where \textbf{every matrix $M_i$ and $C$ are positive semidefinite, while the $R_i$ are \emph{arbitrary symmetric} matrices}.
The vectors $\vec{h_i}$ are in $\R^q$. We denote by $H$ the $q\times l$ matrix formed by the columns $\vec{h_1},\ldots,\vec{h_l}$. 
The Lagrangian dual of Problem~\eqref{SDPPackingExt} is:
  \begin{align} \tag{$D_{\scriptscriptstyle{\textrm{CMB}}}$}
 \inf_{\vec{\mu}\geq0} &\quad \vec{b}^T\vec{\mu}  \label{SDPPackingDualExt}\\
 \textrm{s.t.} &\quad  \sum_{i=1}^l \mu_i M_i \succeq C, \nonumber \\
 &\quad  R_0+ \sum_{i=1}^l \mu_i R_i \preceq 0. \nonumber\\
 &\quad \vec{h_0}+H\vec{\mu}=\vec{0}. \nonumber
\end{align}

We have seen in Section~\ref{sec:mainresult} that
the feasibility of both the primal~\eqref{SDPPacking} and the dual~\eqref{SDPPackingDual}
is sufficient to guarantee that Problem~\eqref{SDPPacking} has a solution
of rank at most $r:=\rank\ C$. For \emph{combined} problems however, the 
feasibility of the couple of programs~\eqref{SDPPackingExt}--\eqref{SDPPackingDualExt}
is not sufficient to guarantee the existence of a solution $(X,Y,\vec{\lambda})$
of Problem~\eqref{SDPPackingExt} in which $\rank\ X\leq r$.
We give indeed an example (Example~\ref{ex:AsyOpt})
where the optimum in Problem~\eqref{SDPPackingExt} is not even attained.
However, we show in the next theorem that an asymptotic result
subsists. Moreover, we shall see in Theorem~\ref{theo:SDPSol1Ext} that a solution
in which $X$ is of rank at most $r$
exists as soon as an additional condition holds (strict dual feasibility).
The proof of Theorem~\ref{theo:SDPSol1Ext} essentially mimics that of Theorem~\ref{theo:SDPSol1}
and is therefore proved in Appendix~A. Theorem~\ref{theo:SDPSol1Asy} turns out to be a consequence of
Theorem~\ref{theo:SDPSol1Ext} and is proved in Appendix~B.

\begin{theorem}  \label{theo:SDPSol1Asy}
We assume that Problems~\eqref{SDPPackingExt} and~\eqref{SDPPackingDualExt}
are feasible. If $\rank\ C=r$,
then there exists a sequence of feasible primal variables
$(X_k,Y_k,\vec{\lambda}_k)_{k\in\mathbb{N}}$
such that $\rank\ X_k \leq r$ for all $k\in\mathbb{N}$ and 
$\langle C,X_k \rangle + \langle R_0,Y_k \rangle +\vec{h_0}^T\vec{\lambda}_k$ converges to the optimum
of Problem~\eqref{SDPPackingExt} as $k\to\infty$.
\end{theorem}

\begin{theorem}  \label{theo:SDPSol1Ext}
We assume that Problem~\eqref{SDPPackingExt} is feasible, and a refined Slater condition
holds for Problem~\eqref{SDPPackingDualExt}, i.e. there is a feasible dual variable which
strictly satisfies the non-affine constraints:
$$\exists \vec{\mub}\geq\vec{0}:\ \sum_i \mub_i M_i \succ C,\ R_0+\sum_i \mub_i R_i\prec 0,\
\vec{h_0}+H\vec{\mub}=\vec{0}.$$
If $\rank\ C=r$, then Problem~\eqref{SDPPackingExt} has a solution $(X,Y,\vec{\lambda})$ in which
$\rank\ X\leq r$. Moreover, if $C\neq 0$, then every solution $(X,Y,\vec{\lambda})$
of Problem~\eqref{SDPPackingExt}
is such that  $\rank\ X\leq n-\overline{r}+r$, where $\overline{r}:=\displaystyle{\min_{i\in[l]}}\ \rank\ M_i$.
\end{theorem}

\begin{example} \label{ex:AsyOpt}
 Consider the following \emph{combined} semidefinite packing problem:
\begin{align} \label{exampSDPP}
\sup_{X \in\mathbb{S}_2^+,\ \vec{\lambda}\in \mathbb{R}^2} &\quad
\frac{3}{100} \left\langle \twotwomat{81}{9}{9}{1},X \right\rangle - \lambda_1-3\lambda_2\\
 \textrm{s.t.} 	& \quad 0 \leq 1+\lambda_1 \nonumber\\
		& \quad X_{1,1} \leq 1+\lambda_2 \nonumber\\
		& \quad X_{2,2} \leq 1+3\lambda_1+\lambda_2 \nonumber.
\end{align}
This problem is in the form of~\eqref{SDPPackingExt} indeed, with $C=\vec{cc}^T$,
$\vec{c}=\frac{\sqrt{3}}{10}[\,9\quad 1]^T$, $\vec{h_0}=[\,\textrm{--}1\quad \textrm{--}3]^T,$
$$M_1=0,\ M_2=\twotwomat{1}{0}{0}{0},\ M_3=\twotwomat{0}{0}{0}{1} \quad\mathrm{and}\quad H= \left(
\begin{array}{ccc}
1 & 0 & 3\\
0 & 1 & 1 
\end{array} \right).$$
Problem~\eqref{exampSDPP} is clearly feasible (e.g.\ for $X=0$, $\vec{\lambda}=\vec{0}$), and the reader
can verify that $\vec{\mu}=\frac{1}{10}[\,1\!\quad 27\quad3]^T$ is dual feasible (in fact, this
is the only dual feasible vector, and hence the dual problem does not satisfy the Slater constraints qualification).
The value of the optimum
is $\frac{31}{10}$, and can be approached arbitrarily closely for the sequence of  
feasible variables $(\vec{x}_k \vec{x}_k^T, \vec{\lambda_k})_{k\in\mathbb{N}}$, where for all $k\geq0$,
$\vec{x}_k=[\,\sqrt{3+k}\quad\sqrt{k}]^T$, $\vec{\lambda}_k=[\,\textrm{--}1\quad k+2]^T$, while
this optimum is not attained by any couple $(X,\vec{\lambda})$ of (bounded) feasible variables.
\end{example}

As in the previous section, we have a result of reduction to a SOCP, which holds when $C$ is of rank $1$, every $R_i=0$ and
$\vec{h_0}=\vec{0}$. Recall that $H$ denotes the matrix formed by the columns $\vec{h_1},\ldots,\vec{h_l}$.

\begin{corollary}\label{coro:packingExtSol1}
Consider the following ``combined'' semidefinite packing problem:
\begin{align}
 \sup_{X\in\mathbb{S}_n,\ \vec{\lambda}\in \R^{q}} &\quad \langle C,X \rangle \label{SDPPackingExt1}\\
 \textrm{s.t.} &\quad \langle M_i,X \rangle \leq \vec{h_i}^T \vec{\lambda} + b_i, \qquad i \in [l],\nonumber\\
 & \quad X\succeq0. \nonumber
\end{align}
Assume that $C=\vec{cc^T}$ has rank $1$. If Problem~\eqref{SDPPackingExt1} and its
Lagrangian dual are feasible, i.e.\
\begin{itemize}
 \item[(i)] $\exists \vec{\lb} \in \R^q:\ H^T \vec{\lb}+\vec{b}\geq{0}$;
 \item[(ii)] $\exists \vec{\mub}\geq\vec{0}:\ \sum_i \mub_i M_i \succeq C,\ \vec{h_0}+H\vec{\mub}=\vec{0}$,
\end{itemize}
then, Problem~\eqref{SDPPackingExt1} is bounded, and its optimal value is the square
of the optimal value of the following SOCP:
\begin{align}
 \sup_{\vec{x}\in \R^{n},\ \vec{\lambda}\in\R^{q}} &\quad \vec{c}^T \vec{x} \label{SOCPPackingExt1}\\
 \textrm{s.t.} &\quad \left\Vert \vectwo{2A_i \vec{x}}{\vec{h_i}^T \vec{\lambda} + b_i-1} \right\Vert_2 \leq \vec{h_i}^T \vec{\lambda}
+ b_i+1, \qquad i \in [l],\nonumber
\end{align}
where the matrices $A_i$ are such that $M_i=A_i^T A_i$. Moreover, if $(\vec{x},\vec{\lambda})$ is a solution of Problem~\eqref{SOCPPackingExt1},
then $(\vec{x x^{T}},\vec{\lambda})$ is a solution of Problem~\eqref{SDPPackingExt1}, and the optimal value of~\eqref{SDPPackingExt1} is
$(\vec{c}^T \vec{x})^2$.
\end{corollary}

\begin{proof}
Theorem~\ref{theo:SDPSol1Asy} guarantees the existence of a sequence of
feasible variables $(X_k,\vec{\lambda}_k)_{k\in{\mathbb{N}}}$ in which  
$X_k$ has rank $1$, i.e.\ $X_k=\vec{x_k x_k}^{T}$, and $\langle C,X_k \rangle=(\vec{c}^T\vec{x}_k)^2$
converges to the optimum of Problem~\eqref{SDPPackingExt1}. This optimal value is therefore equal to
the supremum of $(\vec{c}^T\vec{x})^2$, over all the pairs of vectors
$(\vec{x},\vec{\lambda})\in\R^n\times\R^q$ such that $(\vec{xx}^T,\vec{\lambda})$ is feasible
for Problem~\eqref{SDPPackingExt1}. As in the proof of Corollary~\ref{coro:SOCP}, we notice that
if  $(\vec{x}\vec{x}^T,\vec{\lambda})$ is feasible for Problem~\eqref{SDPPackingExt1}, so is 
$((-\vec{x})(-\vec{x})^T,\vec{\lambda})$, hence we can remove the  square in the objective function.

The SOCP~\eqref{SOCPPackingExt1} is simply obtained from~\eqref{SDPPackingExt1} by substituting $\vec{xx^T}$ from $X$
and $A_i^TA_i$ from $M_i$. We also used the fact that for any vector $\vec{z}$ and for any scalar $\alpha$, the hyperbolic constraint
$$\Vert \vec{z} \Vert_2^2 \leq \alpha$$
is equivalent to the second order cone constraint
$$\left\Vert \vectwo{2\vec{z}}{\alpha-1} \right\Vert_2 \leq \alpha+1.$$
\end{proof}

\paragraph{Application: $\vec{c}-$optimal design of experiments with multiple resource constraints}
~\\
In a more general setting than the classical $\vec{c}-$optimal design problem~\eqref{cOpt} presented in the previous section,
$\vec{w}$ no longer represents the percentage of experimental effort to spend on each experiment, but describes some resource
allocation to the available experiments, that is subject to multiple linear constraints $P\vec{w}\leq\vec{d}$,
where $P$ is a $q\times l$ matrix with nonnegative entries
and $\vec{d}$ is a $q\times1$ vector. This problem arises for example
in a network-wide optimal sampling problem~\cite{SagnolGB10ITC}, where $\vec{w}$ is the vector of the sampling rates of the monitoring devices
on all links of the network, and is subject to linear constraints that limit the overhead of the routers.
We will show that this
problem is a ``combined'' semidefinite packing problem which reduces to an SOCP. The resource constrained $\vec{c}-$optimal design
problem reads as follows:
\begin{align}
 \inf_{\vec{w}\geq\vec{0}}\quad& \vec{c^T} (\sum_i  w_i M_i)^\dagger \vec{c}  \label{cOptExt}\\
\textrm{s.t.}\quad&P\vec{w}\leq\vec{d}. \nonumber
\end{align}
We assume that the optimal design problem is feasible, i.e.\ there exists a vector $\vec{\hat{w}}\geq\vec{0}$ such that
$P\vec{\hat{w}}\leq\vec{d}$ and $\vec{c}$ is
in the range of $\sum_i \hat{w}_i M_i$. Note that we can assume without loss of generality that $\vec{\hat{w}}>\vec{0}$. Otherwise, this would mean
that the constraints $P\vec{w}\leq\vec{d},\ \vec{w}\geq\vec{0}$ force the equality $w_i=0$ to hold for some coordinate $i\in[l]$,
and in this case we could simply remove the experiment $i$ from the set of available experiments.

We can now express the latter problem as an SDP thanks to the Schur complement lemma:
\begin{align}
\inf_{t\in \R,\ \vec{w}\geq\vec{0}} &\quad t \label{SDPPackingDualSC}\\
\textrm{s.t.} &\quad \left( \begin{array} {c|c} \sum_i w_i M_i & \vec{c}\\
                      			\hline
				 \vec{c}^{{}_T} & t
                     \end{array} \right) \succeq 0. \nonumber \\
& \quad P \vec{w} \leq \vec{d}.\nonumber
\end{align}
Since the optimal $t$ is positive (we exclude the trivial case $\vec{c}=\vec{0}$), the latter matrix inequality may be rewritten as
$$\sum_i w_i M_i \succeq \frac{\vec{cc^T}}{t},$$
by using the Schur complement lemma again. Finally, we make the change of variables $\vec{\mu}=t\vec{w}$
and Problem~\eqref{SDPPackingDualSC} is equivalent to
\begin{align}
\inf_{\vec{\mu}\geq\vec{0},t\geq0} &\quad t \label{SDPPackingDual3}\\
\textrm{s.t.} &\quad \sum_{i=1}^l \mu_i M_i \succeq \vec{cc^T}\nonumber \\
& \quad P\vec{\mu} \leq t\vec{d}.\nonumber
\end{align}

This problem is exactly in the form of Problem~\eqref{SDPPackingDualExt},
for  $C=\vec{cc^T}$, $\mu_{l+1}=t,\ \vec{b}=[0,\ldots,0,1]^T \in \R^{l+1},\ M_{l+1}=0,\ \vec{h_0}=\vec{0},\ H=[P,-\vec{d}] $, 
and for all $i\in 0,\ldots,l+1,\ R_i=0$ (we also need to introduce a nonnegative slack variable to handle the inequalities as equalities).

Let $\lambda:=\vec{c}^T (\sum_i M_i)^\dagger \vec{c}^T$, so that $\lambda\sum_i M_i \succeq \vec{cc}^T$.
We set $\overline{t}=\max_{i\in[l]} (\lambda/\hat{w}_i)$
($\overline{t}$ is well defined because $\vec{\hat{w}}>\vec{0}$). The vector
$\vec{\mub}:=\overline{t} \vec{\hat{w}}$ is dual feasible, because
$P\vec{\mub}\leq \overline{t} \vec{d},$ and
$\sum_{i=1}^l \mub_i M_i \succeq \lambda  \sum_{i=1}^l M_i \succeq \vec{cc^T}.$
In addition,
the corresponding primal problem is clearly feasible (for $\vec{\lambda}=\vec{0}$, since $\vec{b}\geq\vec{0}$),
and thus we can use Corollary~\ref{coro:packingExtSol1}: the $\vec{c}-$optimal design problem with resource constraints~\eqref{cOptExt}
reduces to the SOCP~\eqref{SOCPPackingExt1}. 
We give below this SOCP (with the parameters $\vec{b}$, $M_i$, $H$ and the slacks defined as above),
as well as its dual:

\noindent\ \begin{minipage}[t]{0.49\textwidth}
 \begin{align*}
  \sup_{\raisebox{-0.1cm}[0cm][-0.2cm]{\ensuremath{\substack{\vec{x}\in\R^n\\ {\vec{\lambda}\in\R^q}}}}}
&\quad \vec{c}^T \vec{x}\\
&\quad  \left\Vert \vectwo{2A_i \vec{x}}{\vec{p_i}^T \vec{\lambda} -1} \right\Vert_2 \leq \vec{p_i}^T \vec{\lambda} +1\quad (\forall i\in[l]),\\
&\quad \vec{d}^T \vec{\lambda}\leq 1,\\
& \quad \vec{\lambda}\geq\vec{0}.
 \end{align*}
\end{minipage}
\begin{minipage}[t]{0.49\textwidth}
 \vspace{-0.4cm}
 \begin{align*}
\qquad \inf_{\raisebox{-0.3cm}[0cm][-0.2cm]{\ensuremath{\substack{\vec{\mu}\geq\vec{0},t \geq0\\\vec{\alpha}\geq\vec{0}, (\vec{z_i})_{i\in[l]}}}}}
&\quad \sum_{i=1}^l \alpha_i +t\\
&\quad \sum_{i=1}^l A_i^T \vec{z_i}=\vec{c},\\
& \quad P \vec{\mu} \leq t \vec{d},\\
& \quad  \left\Vert \vectwo{\vec{z_i}}{\alpha_i-\mu_i} \right\Vert_2 \leq \alpha_i+\mu_i\\
& \quad \phantom{\left\Vert \vectwo{\vec{z_i}}{\alpha_i-\mu_i} \right\Vert_2 \leq}
\raisebox{0.4cm}[0cm][0cm]{\ensuremath{(\forall i \in [l]),}}
 \end{align*}
\end{minipage}

\vspace{0.3cm}
\noindent where the vectors $\vec{p_1},\ldots,\vec{p_l}\in\R^q$ are the columns of the matrix $P$, and
for all $i\in[l]$, $A_i$ is such that $A_i^TA_i=M_i$. The dual problem satisfies the (refined)
Slater condition,
because $\vec{c}\in\range (\sum_i M_i)=\sum_i \range(A_i^T)$, so that
$\exists \vec{\zb_1},\ldots,\vec{\zb_l}: \sum_{i=1}^l A_i^T \vec{\zb_i}=\vec{c}$,
$P \vec{\mub} \leq \overline{t} \vec{d}$ and for $\vec{\overline{\alpha}}>\vec{0}$ large enough, the non-affine cone
constraints are satisfied with a strict inequality. Hence, strong duality holds and the values of these
two problems are equal.
By construction, the optimal design variable $\vec{w}$
is related to the dual optimal variables $\vec{\mu}$ and $t$ by the relation $\vec{w}=t^{-1} \vec{\mu}$.
Moreover, Corollary~\ref{coro:packingExtSol1} shows that the optimal value of Problem~\eqref{cOptExt}
is the square of the optimal value of these SOCPs.

\section{Proofs of the theorems}
\label{sec:proof}
\begin{proof}[Proof of Theorem~\ref{theo:feasibility-Boundness}]
The fact that Problem~\eqref{SDPPacking} is feasible if and only if every $b_i$ is nonnegative is clear, since $X=0$ is always feasible in this case and $M_i \succeq 0, X \succeq 0$, implies $\langle M_i,X \rangle \geq 0$.

Now, we assume that each $b_i$ is nonnegative, and we show that Problem~\eqref{SDPPacking} is bounded if
and only if $\range C \subset \range \sum_i M_i$.
The positive semidefiniteness of the matrices $M_i$ implies that there exists matrices $A_i$ ($i\in[l]$) such that
$A_i^T A_i=M_i$, and $[A_1^T,\cdots,A_l^T] [A_1^T,\cdots,A_l^T]^T=\sum_i M_i$. We also consider a decomposition
$C=\sum_{k=1}^r \vec{c_k} \vec{c_k}^T$. For any factorization $M=A^T A$ of a positive semidefinite matrix $M$,
it is known that $\range M = \range A$, and so the following
equivalence relations hold:
\begin{align}
\range C \subset \range \sum_i M_i & \Longleftrightarrow \forall k\in[r],\ \vec{c_k} \in \range(\sum_i M_i)=\range ([A_1^T,\cdots,A_l^T])  \nonumber \\
& \Longleftrightarrow \forall k\in[r],\ \vec{c_k}\in \left( \bigcap_{i=1}^l \nullspace (A_i)\right)^\bot. \label{equivRange2}
\end{align}

We first assume that the range inclusion condition does not hold. Relation~\eqref{equivRange2} shows that
$$\exists k \in [r], \exists \vec{h} \in \R^n: \forall i\in[l],\quad A_i \vec{h}=0,\quad \vec{c_k}^T \vec{h}\neq0.$$
Now, notice that $X=\alpha \vec{h} \vec{h}^T$ is feasible for all $\alpha>0$, since $\alpha \langle A_i^T A_i , \vec{h} \vec{h}^T \rangle=0\leq b_i$.
This contradicts the fact that Problem~\eqref{SDPPacking} is bounded, because $\langle C,X \rangle \geq \alpha (\vec{c_k}^T \vec{h})^2,$ and $\alpha$
can be chosen arbitrarily large.

Conversely, if the range inclusion holds, we consider the Lagrangian dual~\eqref{SDPPackingDual}
of Problem~\eqref{SDPPacking}:
The range inclusion condition indicates that this problem is feasible, because it implies the existence of a scalar $\lambda>0$ such
that $\lambda \sum_i M_i \succeq C$ (we point out that a convenient value for
$\lambda$ is $\sum_{k=1}^r \vec{c_k}^T (\sum_i M_i)^\dagger \vec{c_k}$; this can  be seen with the help of the Schur complement lemma).
This means that Problem~\eqref{SDPPackingDual} has a finite optimal value $OPT\leq \lambda \sum_i b_i$, and by weak duality,
Problem~\eqref{SDPPacking} is bounded (its optimal value cannot exceed $OPT$).
%
%
%
\end{proof}

Before proving Theorem~\ref{theo:SDPSol1}, we need to show that we can project Problem~\eqref{SDPPacking}
on a subspace such that the projected problem~\eqref{Pprime} and its Lagrangian dual are
strictly feasible (Proposition~\ref{prop:reduction}).

\begin{proof}[Proof of Proposition~\ref{prop:reduction}]

Let $\mathcal{I}_0, \mathcal{I}, U$ and $V$ be defined as in the paragraph
preceding the statement of the proposition. 
Note that every matrix $M_i$
can be decomposed as $M_i=U \tilde{M}_i U^T$ for a given matrix $\tilde{M}_i$, because its range is included in the range
of $\sum_i M_i$ (we have $\tilde{M}_i=U^T M_i U$).
The same observation holds for $C$, which can be decomposed as $C=U \tilde{C} U^T$ (we have assumed the range inclusion 
$\range C \subset \range \sum_i M_i$).
 Hence, Problem~\eqref{SDPPacking} is equivalent to:
\begin{align*}
 \max_{X\succeq 0} &\quad \langle \tilde{C}, U^T X U \rangle \\
 \textrm{s.t.} &\quad \langle \tilde{M_i} , U^T X U\rangle \leq b_i, \qquad i \in [l].
\end{align*}
After the change of variable $Z_0=U^T X U$ ($Z_0$ is a positive semidefinite matrix
if $X$ is), we obtain a reduced semidefinite packing problem
\begin{align}
 \max_{Z_0\succeq 0} &\quad \langle \tilde{C}, Z_0 \rangle \label{SDPPackingsumMi}\\
 \textrm{s.t.} &\quad \langle \tilde{M_i} , Z_0\rangle \leq b_i, \qquad i \in [l].\nonumber
\end{align}
By construction, if $Z_0$ is a solution of~\eqref{SDPPackingsumMi}, then
$X:=UZ_0U^T$ is a solution of~\eqref{SDPPacking}.
Note that the  projected matrices in the constraints now
satisfy $\sum_i \tilde{M_i}=U^T (\sum_i M_i) U\succ0$.

We shall now consider a second projection, in order to get rid of the constraints
in which $b_i=0$. Note that each constraint indexed by $i\in\mathcal{I}_0$ is equivalent to
imposing that $Z_0$ belong to the nullspace of the matrix $\tilde{M}_i$. Since the columns of $V$
form a basis of $\cap_{i\in\mathcal{I}_0} \nullspace \tilde{M}_i$, any semidefinite matrix
$Z_0$ which is feasible for Problem~\eqref{SDPPackingsumMi} must be of the form $V Z V^T$
for some positive semidefinite matrix $Z$.
Hence, Problem~\eqref{SDPPackingsumMi} reduces to:
\begin{align}
  \max_{Z \succeq 0} &\quad \langle V^T \tilde{C} V ,Z \rangle \label{SDPPackingbi}\\
  \textrm{s.t.} &\quad \langle V^T \tilde{M}_i V ,Z \rangle \leq b_i, \qquad i \in \mathcal{I}.\nonumber
\end{align}
which is nothing but Problem~\eqref{Pprime}, because $V^T\tilde{M}_i V=V^TU^TM_iUV=M_i'$
and  $V^T\tilde{C} V=C'$. By construction,
If $Z$ is a solution of~\eqref{SDPPackingbi}$\equiv$\eqref{Pprime}, then $VZV^T$ is a solution of ~\eqref{SDPPackingsumMi},
and $(UV)Z(UV)^T$ is a solution of the original problem~\eqref{SDPPacking}. This proves the point $(iii)$
of the proposition.

We have pointed out above that  $\sum_i \tilde{M_i}\succ0$. Therefore, there exists 
a real $\lambda>0$ such that $\lambda \sum_i \tilde{M}_i \succ \tilde{C}$, and 
$\lambda \sum_i M_i' = V^T \big(\lambda  \sum_i \tilde{M}_i\big) V \succ V^T \tilde{C} V = C'$. This proves the strict
dual feasibility of Problem~\eqref{Pprime} (point $(ii)$ of the proposition).
Finally, since every $b_i$ is positive for $i\in\mathcal{I}$, it is clear that the matrix $\overline{Z}=\varepsilon I \succ0$
is strictly feasible for Problem~\eqref{Pprime} as soon as $\varepsilon>0$ is sufficiently small. This establishes
the point $(i)$, and the proposition is proved.

\end{proof}

We can now prove the main result of this article.
We will first show that the result holds when every $M_i$ is positive definite, thanks to the complementary slackness relation. Then, 
the general result is
obtained by continuity. We point out at the end of this section the sketch of an alternative
proof of Theorem~\ref{theo:SDPSol1} for the case in which $r=1$, based on the bidual of Problem~\eqref{SDPPacking} and Schur complements,
that shows directly that Problem~\eqref{SDPPacking} reduces to the SOCP~\eqref{SOCPPacking}.

\begin{proof}[Proof of Theorem~\ref{theo:SDPSol1}]
We will show that the result of the theorem holds for any semidefinite packing problem which
is strictly feasible, and whose dual is strictly feasible. Then, by Proposition~\ref{prop:reduction},
we can say that Problem~\eqref{Pprime} has a solution $Z$ of rank at most $r':=\rank\ C'$, and
$X:=(UV)^TZ(UV)$ is a solution of the original problem which is of rank at most $r'\leq r$.

So let us assume without loss of generality that~\eqref{SDPPacking} and~\eqref{SDPPackingDual}
are strictly feasible:
$$\forall i\in[l], b_i>0\quad\textrm{and}\quad \exists \lambda>0: \lambda \sum_i M_i \succ C.$$
The Slater condition is fulfilled for this pair of programs, and so strong duality holds (the optimal
value of~\eqref{SDPPacking} equals the optimal value of~\eqref{SDPPackingDual}, and the dual problem attains
its optimum. In addition, the strict dual feasibility implies that~\eqref{SDPPacking} also attains its
optimum. The pairs of primal and dual solutions $(X^*,\vec{\mu}^*)$ are characterized
by the Karush-Kuhn-Tucker (KKT) conditions:
\begin{align*}
\textrm{Primal Feasibility:} &\qquad \forall i \in [l],\quad \langle M_i, X^* \rangle \leq b_i;\\
& \qquad X^*\succeq 0;\\
\textrm{Dual Feasibility:} &\qquad \vec{\mu^*}\geq0,\quad  \sum_{i=1}^l \mu_i^* M_i  \succeq C;\\
\textrm{Complementary Slackness:} &\qquad (\sum_{i=1}^l \mu_i^* M_i - C)\ X^*=0,\\
&\qquad \forall i\in[l],\ \mu_i^*(b_i-\langle M_i, X^*\rangle)=0.
\end{align*}

Now, we consider the case in which $M_i\succ0$ for all $i$, and we choose an arbitrary pair of primal and dual optimal solutions
$(X^*,\vec{\mu}^*)$. The dual feasibility relation implies $\vec{\mu}^*\neq \vec{0}$, and so
$\sum_i \mu_i^* M_i$ is a positive definite matrix (we exclude the trivial case $C=0$). Since $C$ is of rank $r$,
we deduce that $$\operatorname{rank}(\sum_i \mu_i^* M_i - C) \geq n-r.$$
Finally, the complementary slackness relation indicates that the columns of $X^*$ belong to the nullspace of $(\sum_i \mu_i^* M_i-C)$, which is a vector space of dimension at most $n-(n-r)=r$,
and so we conclude that $\operatorname{rank} X^* \leq r.$\par
\vspace{12pt}

We now turn to the study of the general case in which $M_i\succeq0$.
 To this end, we consider the perturbed problems
\begin{align}
  \max &\quad \langle C,X \rangle  \nonumber \\
\textrm{s.t.}&\quad \langle M_i + \varepsilon I,X \rangle \leq b_i \label{Pepsilon} \tag{$P_\varepsilon$} \\
&\quad X \succeq 0,\nonumber
\end{align}
and
\begin{align}
 \min_{\vec{\mu}\geq0} &\quad \sum_{i=1}^l \mu_i b_i, \label{Depsilon} \tag{$D_\varepsilon$}\\
 \textrm{s.t.} &\quad  \sum_{i=1}^l \mu_i (M_i+\varepsilon I) \succeq C \nonumber.
\end{align}
where $\varepsilon\geq0$. Note that the strict feasibility of the unperturbed problems~\eqref{SDPPacking}
and~\eqref{SDPPackingDual} implies that of~\eqref{Pepsilon} and~\eqref{Depsilon} on a neighborhood $\varepsilon \in [0,\varepsilon_0]$,
$\varepsilon_0>0$. We denote by $(X^\varepsilon,\vec{\mu^\varepsilon})$
a pair of primal and dual solutions of~\eqref{Pepsilon}--\eqref{Depsilon}.

If $\varepsilon>0$,  $M_i + \varepsilon I \succ 0$ and it follows
from the previous discussion that $X^\varepsilon$ is of rank at most $r$.
We show below that we can choose the optimal variables
$(X^\varepsilon, \vec{\mu^\varepsilon})_{\varepsilon \in ]0,\varepsilon_0]}$
within a bounded region, so that we can construct a converging
subsequence $(X^{\varepsilon_k},\vec{\mu^{\varepsilon_k}})_{k \in \mathbb{N}},\ \varepsilon_k\to0$
from these variables.
To conclude, we will see that the limit $(X^0,\vec{\mu^0})$ satisfies the
KKT conditions for Problems~\eqref{SDPPacking}--\eqref{SDPPackingDual},
and that $X^0$ is of rank at most~$r$.

Let us denote the optimal value of Problems~\eqref{Pepsilon}--\eqref{Depsilon} by $OPT(\varepsilon).$ Since the constraints
of the primal problem becomes tighter when $\varepsilon$ grows, it is clear that $OPT(\varepsilon)$ is nonincreasing
with respect to $\varepsilon$,
so that $$\forall \varepsilon \in [0,\varepsilon_0],\  OPT(\varepsilon_0)\leq OPT(\varepsilon)\leq OPT(0).$$

We have:
$$\lambda (\sum_i M_i + \varepsilon I) -C \succ \lambda (\sum_i M_i) -C ,$$
and so we can write
\begin{align*}
 \langle \lambda \sum_i M_i -C, X^\varepsilon \big\rangle
& \leq
\langle \lambda \sum_i (M_i+\varepsilon I) -C, X^\varepsilon \big\rangle\\
 &=\lambda \langle \sum_i (M_i+\varepsilon I) , X^\varepsilon \big\rangle   
-OPT(\varepsilon)\\
& \leq  \lambda \sum_i b_i -OPT(\varepsilon_0)
\end{align*}
where the equality comes from the expression of $OPT(\varepsilon)$ and the latter inequality
follows from the constraints of the Problem~\eqref{Pepsilon}. The matrix $\lambda \sum_i M_i -C$ is positive definite
by assumption and its smallest eigenvalue $\lambda'$ is therefore positive. Hence,
$$\lambda'\ \trace\ X^\varepsilon \leq \langle \lambda \sum_i M_i -C, X^\varepsilon \big\rangle \leq \vec{\mub}^T \vec{b} -OPT(\varepsilon) \leq \lambda \sum_i b_i -OPT(\varepsilon_0).$$
This shows that the positive semidefinite matrix $X^\varepsilon$ has its trace bounded, and therefore
all its entries are bounded.

It remains to show that the dual optimal variable $\vec{\mu^\varepsilon}\geq\vec{0}$ is bounded.
This is simply done by writing:
$$ \forall i\in[l],\quad b_i \mu_i^\varepsilon \leq \vec{b}^T \vec{\mu^\varepsilon} = OPT(\varepsilon) \leq OPT(0).$$
By assumption, $b_i>0$, and the entries of the vector $\vec{\mu^\varepsilon}\geq \vec{0}$ are bounded.

\sloppypar{~\\ \indent We can therefore construct a
sequence of pairs of primal and dual optimal solutions
$(X^{\varepsilon}, \vec{\mu^{\varepsilon_k}})_{k \in \mathbb{N}}$
that converges, with $\varepsilon_k \underset{k \to \infty}{\longrightarrow} 0$, $\varepsilon_k>0$.}
The limit $X^0$ of this sequence is of rank at most $r$, because the rank is a lower semicontinuous function and $\rank\ X^{\varepsilon_k}\leq r$
for all $k\in\mathbb{N}$.
It remains to show that $X^0$ is a solution of Problem~\eqref{SDPPacking}. The $\varepsilon-$perturbed KKT conditions must
hold for all $k\in\mathbb{N}$, and so they hold for the pair $(X_0, \vec{\mu^{0}})$
by taking the limit (the limit of any sequence of positive semidefinite matrices is a positive semidefinite matrix because $\mathbb{S}_n^+$ is closed).
This concludes the proof.

\end{proof}

\paragraph{Sketch of an alternative proof of Theorem~\ref{theo:SDPSol1} when \boldmath $r=1$ \unboldmath}
~\\
By Proposition~\ref{prop:reduction}, we only need to show that the result holds for the reduced
problem~\eqref{Pprime}, and so we assume without loss of generality that strong duality holds for
all the optimization problems considered below.

When $r=1$, there is a vector $\vec{c}$ such that $C=\vec{cc^T}$ and the dual problem of~\eqref{SDPPacking} takes the form:
 \begin{align}
 \min_{\vec{\mu}\geq0} &\quad \vec{\mu^T b} \label{SDPPackingDualc}\\
 \textrm{s.t.} &\quad \vec{cc^T} \preceq \sum_i \mu_i M_i. \nonumber
\end{align}
Now, setting $t=\vec{\mu^T b}$, and $\vec{w}=\frac{\vec{\mu}}{t}$, so that the new variable $\vec{w}$ satisfies
$\vec{w^Tb}=1$, the constraint of the previous problem
becomes $\frac{\vec{cc^t}}{t} \preceq \sum_i w_i M_i$. This matrix inequality, together with the fact that the optimal $t$ is positive, 
can be reformulated thanks to the Schur
complement lemma, and~\eqref{SDPPackingDualc} is equivalent to:
 \begin{align}
 \min_{t\in \R,\vec{w}\geq\vec{0}} &\quad t \label{SDPPackingDual2}\\
 \textrm{s.t.} &\quad \left( \begin{array} {c|c} \sum_i w_i M_i & \vec{c}\\
                       			\hline
					 \vec{c^{{}_T}} & t
                      \end{array} \right) \succeq 0. \nonumber \\
&\quad \vec{w^Tb}=1. \nonumber
\end{align}
We dualize this SDP once again to obtain the bidual of Program~\eqref{SDPPacking} (strong duality holds):
 \begin{align}
 \max_{\beta \in \R,Z \in \mathbb{S}^+_{n+1}} &\quad -\beta-2\vec{v^Tc} \label{SDPbidual}\\
 \textrm{s.t.} &\quad \langle W, M_i \rangle \leq \beta b_i, \quad i\in[l]  \nonumber \\
&\quad Z=\left( \begin{array} {c|c} W & \vec{v}\\
                     \hline
			\vec{v^{{}_T}} & 1
                      \end{array} \right) \succeq 0.
 \nonumber
\end{align}
We notice that the last matrix inequality is equivalent to $W \succeq \vec{v v^T}$, using a Schur complement.
Since $M_i\succeq0$, we can assume that $W=\vec{vv^T}$ without loss of generality, and~\eqref{SDPbidual} becomes:
 \begin{align}
 \max_{\beta \in \R,\vec{v} \in \R^n} &\quad -\beta-2\vec{v^Tc} \label{SOCPbidual}\\
 \textrm{s.t.} &\quad \Vert A_i \vec{v} \Vert^2 \leq \beta b_i, \quad i=1\in[l],  \nonumber
\end{align}
where $A_i$ is a matrix such that $A_i^TA_i=M_i$.

We now define the new variables $\alpha=\sqrt{\beta}$, and
$\vec{x}=\frac{\vec{v}}{\alpha}$, so that \eqref{SOCPbidual} becomes:
 \begin{align}
 \max_{\vec{x} \in \R^n}  &\quad \left( \max_\alpha -\alpha^2-2\alpha \vec{x^Tc} \right) \label{SOCPfenchel}\\
 \textrm{s.t.} &\quad \Vert A_i \vec{x} \Vert \leq \sqrt{b_i}, \quad i=1\in[l].  \nonumber
\end{align}
The reader can finally verify that the value of the max within parenthesis is $(\vec{c^T x})^2$, and we have proved
that the SDP~\eqref{SDPPacking} reduces to the SOCP~\eqref{SOCPPacking}.
By the way, this guarantees that the SDP~\eqref{SDPPacking} has a rank-one solution.\qed

\section*{Acknowledgment}
The author thanks St\'ephane Gaubert for his useful comments and enlightening
discussions, as well as for his warm support. He also expresses his gratitude to
two anonymous referees. In a previous version, the main result was restricted to the case in which $r=1$.
One referee suggested an alternative proof with an elegant complementary slackness argument, which led to
the more general statement of Theorem~\ref{theo:SDPSol1}.
The author also thanks a second referee for his useful remarks, which helped to clarify the
consequences and the presentation of these results.


\clearpage

\appendix
\numberwithin{equation}{section}
\section{Proof of Theorem~\ref{theo:SDPSol1Ext}}
Before we give the proof of Theorem~\ref{theo:SDPSol1Ext}, we need one additional technical lemma,
which shows that one can assume without loss of generality that the
primal problem is strictly feasible, and that the vector space spanned by the
vectors $\vec{h_0},\vec{h_1},\ldots,\vec{h_l}$ coincides with the cone
generated by the same vectors. One can consider this lemma as
the analog of Proposition~\ref{prop:reduction} for combined problems.
\begin{lemma}
 \label{lem:bipositiflemma}
We assume that the conditions of Theorem~\ref{theo:SDPSol1Ext} are fulfilled.
Then, there exists a subset $\mathcal{I}\subset[l]$, as well as matrices
$C'\succeq0$ and $M_i'\succeq0$ ($i\in\mathcal{I}$), so that the reduced ``combined'' semidefinite packing problem 
$$\max_{Z\succeq0,\ Y\succeq0,\ \vec{\lambda}}\ \langle C', Z \rangle + \langle R_0, Y \rangle + \vec{h_0}^T\vec{\lambda} 
\qquad \textrm{s.t.}\qquad \forall i \in \mathcal{I},\
\langle M_i', Z \rangle \leq b_i + \langle R_i, Y\rangle + \vec{h_i}^T\vec{\lambda}$$
has the same optimal value as~\eqref{SDPPackingExt} and satisfies the following properties:
\begin{itemize}
\item[$(i)$] $\exists (Z'\succ0, Y' \succ0, \vec{\lambda'}):\ \forall i \in \mathcal{I},\ \langle M_i, Z'\rangle < b_i + \langle R_i, Y' \rangle + \vec{h_i}^T \vec{\lambda'}$;
\item[$(ii)$] The cone $K$ generated by the vectors $(\vec{h_i})_{i\in\{0\}\cup\mathcal{I}}$ is a vector space.
\item[$(iii)$] $\rank\ C' \leq \rank\ C$;
\item[$(iv)$] There is a matrix $U$ with orthonormal columns such that if $(Z,Y,\vec{\lambda})$ is a
solution of the reduced problem, then $(X:=UZU^T,Y,\vec{\lambda})$
is a solution of Problem~\eqref{SDPPackingExt} (which of course satisfies $\rank\ X \leq \rank\ Z$).
\end{itemize}
\end{lemma}

\begin{proof}
In this lemma, $(i)$ and $(ii)$ are the properties that we will need to prove Theorem~\ref{theo:SDPSol1Ext}. Properties
$(iii)$ and $(iv)$ ensure that if the theorem holds for the reduced problem, then the result also
holds for the initial problem~\eqref{SDPPackingExt}.
We handle separately the cases in which the initial problem does not satisfy the property $(i)$ or $(ii)$. If both cases arise simultaneously,
we obtain the result of this lemma by applying
successively the following two reductions.

Let $(X^*,Y^*,\vec{\lambda}^*)$ be an optimal solution of Problem~\eqref{SDPPackingExt} ; the existence of a solution is
guaranteed by the (refined) Slater condition satisfied by the dual problem indeed~(see e.g.\ \cite{Roc70,Ber95}). We denote by
$\mathcal{I}_0\subset [l]$  the subset of indices for which  $b_i + \langle R_i, Y^* \rangle + \vec{h_i}^T \vec{\lambda^*}=0$
(note that we have  $b_i + \langle R_i, Y^* \rangle + \vec{h_i}^T \vec{\lambda^*} \geq0 $  for all $i$ because $M_i\succeq0$ implies
$\langle M_i, X^*\rangle\geq0$).
We define $\mathcal{I}:=[l] \setminus \mathcal{I}_0$. In Problem~\eqref{SDPPackingExt}, we can replace 
the constraint $\langle M_i, X \rangle\leq b_i + \langle R_i, Y \rangle+ \vec{h_i}^T \vec{\lambda}$ by $\langle M_i, X \rangle=0$
for all $i\in\mathcal{I}_0$ ,
since $(X^*,Y^*,\vec{\lambda^*})$ satisfies this stronger set of constraints. For a feasible positive semidefinite matrix $X$, 
this implies
$\langle \sum_{i\in\mathcal{I}_0} M_i,  X \rangle=0$, and even $\sum_{i\in\mathcal{I}_0} M_i X =0$. 
Therefore, $X$ is of the form $U Z U^T$ for some positive semidefinite matrix $Z$, where the columns of $U$ form
an orthonormal basis of the nullspace of $M_0:=\sum_{i\in\mathcal{I}_0} M_i$
($U$ is obtained by taking the eigenvectors corresponding to the
vanishing eigenvalues of $M_0$). Hence, Problem~\eqref{SDPPackingExt} is equivalent to:
\begin{align}
 \max &\quad \langle U^T C U ,Z \rangle + \langle R_0, Y \rangle +\vec{h_0}^T\vec{\lambda} \label{SDPPackingExtbi}\\
 \textrm{s.t.} &\quad \langle U^T M_i U ,Z \rangle \leq b_i+ \langle R_i, Y\rangle + \vec{h_i}^T \vec{\lambda}, \qquad i \in \mathcal{I},\nonumber\\
&\quad Z \succeq 0,\ Y\succeq0. \nonumber
\end{align}
We have thus reduced the problem to one for which
$b_i+ \langle R_i, Y^*\rangle +\vec{h_i}^T \vec{\lambda^*}>0$ for all $i$, and strict feasibility follows
(i.e.\ property $(i)$ holds, consider
$\vec{\lambda'}=\vec{\lambda^*}, Y'=Y^*+\eta_1 I $, and $Z'=\eta_2 I$ for sufficiently small reals $\eta_1>0$ and $\eta_2>0$).
Moreover, the projected matrix $C':=U^T C U$ in the objective function
has a smaller rank than $C$ (i.e.\ $(iii)$ holds). Finally, $(iv)$ holds for the reduced problem by construction:
if $(Z,Y,\vec{\lambda})$ is a solution of Problem~\eqref{SDPPackingExtbi}, then $(X:=U Z U^T,Y,\vec{\lambda})$
is a solution of Problem~\eqref{SDPPackingExt}, both problems have the same optimal value, and of course $\rank\ X \leq \rank\ Z$.

We now handle the second case, in which Property $(ii)$ does not hold for Problem~\eqref{SDPPackingExt}.
The set $K=\{\ [\vec{h_0},H]\vec{v},\  \vec{v}\in\R^{l+1}, \vec{v}\geq\vec{0}\}$ is a closed convex cone. Hence, it is known that it can be decomposed
as $K=L+Q$, where $L$ is a vector space and $Q\subset L^\bot$ is a closed convex pointed cone
($L=K\cap (-K)$ is the \emph{lineality space} of $K$).
The interior of the dual cone $Q^*$ is therefore nonempty, i.e.\ $\exists \vec{\lambda}:\forall \vec{q} \in Q\setminus\{\vec{0}\}, \vec{\lambda}^T\vec{q}>0.$
Let $\vec{\lambda_0}$ be the orthogonal projection of $\vec{\lambda}$ on $L^\bot$, so that
$\vec{\lambda_0}^T\vec{q}=\vec{\lambda}^T\vec{q}>0$ for all $\vec{q}\in Q\setminus\{\vec{0}\}$, and $\vec{\lambda_0}^T \vec{x}=0$ for all $\vec{x} \in L$.
Now, we define the set of indices $\mathcal{I}=\{i\in[l]: \vec{h_i} \in L\}$, and its complement $\mathcal{I}_0=[l]\setminus \mathcal{I}$.
For all $i\in\mathcal{I}_0$, $\vec{h_i}=\vec{x_i}+\vec{q_i}$ for a vector $\vec{x_i}\in L$ and a  vector $\vec{q_i}\in Q\setminus\{\vec{0}\}$,
so that $\vec{\lambda_0}^T\vec{h_i}=\vec{\lambda_0}^T\vec{x_i}+\vec{\lambda_0}^T\vec{q_i}=\vec{\lambda_0}^T\vec{q_i}>0$.
For the indices $i\in\mathcal{I}$, it is clear that $\vec{\lambda_0}^T\vec{h_i}=0$. Finally, since $\vec{h_0}+H\vec{\mub}=0$, we have
$-\vec{h_0} \in K$, so that $\vec{h_0}\in L$ and $\vec{h_0}^T \vec{\lambda}=0.$
To sum up, we have proved the existence of a vector $\vec{\lambda_0}$ for which
$$\forall i\in \{0\}\cup \mathcal{I},\ \vec{\lambda_0}^T\vec{h_i}=0\quad \textrm{and}\quad
\forall i\in \mathcal{I}_0, \vec{\lambda_0}^T\vec{h_i}>0.$$ 
Let $(X^*,Y^*,\vec{\lambda}^*)$ be an optimal solution of Problem~\eqref{SDPPackingExt}. 
For all positive real $t$, $(X^*,Y^*,\vec{\lambda}^*+t\vec{\lambda_0})$ is also a solution, because it is feasible and has the same
objective value. Letting $t\to\infty$, we see that the constraints of the problem that are indexed by $i\in\mathcal{I}_0$ may be removed
without changing the optimum. We have thus reduced the problem to one for which $(ii)$ holds.

\end{proof}

We can now prove Theorem~\ref{theo:SDPSol1Ext}.
The proof mimics that of Theorem~\ref{theo:SDPSol1}, i.e. we first show that the result holds
when each $M_i$ is positive definite, and the general result is
obtained by continuity. The only difference is how we show that we can choose optimal variables
$(X^\varepsilon, Y^\varepsilon,\vec{\lambda^\varepsilon},\vec{\mu^\varepsilon})_{\varepsilon \in ]0,\varepsilon_0]}$
for a perturbed problem within a bounded region.

\begin{proof}[Proof of Theorem~\ref{theo:SDPSol1Ext}]
\sloppypar{By Lemma~\ref{lem:bipositiflemma}, we may assume without loss of generality that
$K=\operatorname{cone}\{\vec{h_0},\ldots,\vec{h_l}\} \supset -K$ and
that the primal problem is strictly feasible.}
The strict feasibility of the primal problem ensures that strong
duality holds, i.e. the optimal value of~\eqref{SDPPackingExt} equals the optimal value
of~\eqref{SDPPackingDualExt}, and the optimum is attained in the dual problem. Moreover, the
(refined) Slater constraints qualification for the dual problem guarantees the existence
of primal optimal variables as well~(see e.g.\ Theorem~28.2 in~\cite{Roc70}).
The pairs of primal and dual solutions $\big((X^*,Y^*,\vec{\lambda^*}),\vec{\mu}^*\big)$ are characterized
by the Karush-Kuhn-Tucker (KKT) conditions:
\begin{align*}
 \textrm{Primal Feasibility:} &\qquad \forall i \in [l],\quad \langle M_i, X^* \rangle \leq b_i + \langle R_i, Y^* \rangle + \vec{h_i}^T \vec{\lambda^*} ,\\
 & \qquad X^*\succeq0,\ Y^*\succeq 0;\\
 \textrm{Dual Feasibility:} &\qquad \vec{\mu^*}\geq0,\quad  \sum_{i=1}^l \mu_i^* M_i  \succeq C,\quad
   R_0+\sum_{i=1}^l\mu_i^* R_i \preceq 0,\quad \vec{h_0}+H\vec{\mu^*}=0;\\
 \textrm{Complementary Slackness:} &\qquad (\sum_{i=1}^l \mu_i^* M_i - C)\ X^*=0, \qquad ( R_0 + \sum_{i=1}^l \mu_i^* R_i)\ Y^*=0, \\
&\qquad \forall i\in[l],\ \mu_i^*(b_i+\langle R_i, Y^* \rangle + \vec{h_i}^T \vec{\lambda^*} -\langle M_i, X^*\rangle)=0.
 \end{align*}
 Now, we consider the case in which $M_i\succ0$ for all $i$, and we choose an arbitrary pair of primal and dual optimal solutions
 $\big((X^*,Y^*,\vec{\lambda^*}),\vec{\mu}^*\big)$. The dual feasibility relation implies $\vec{\mu}^*\neq \vec{0}$, and so
 $\sum_i \mu_i^* M_i$ is a positive definite matrix (we exclude the trivial case $C=0$). Since $C$ is of rank $r$,
 we deduce that
 $$\rank(\sum_i \mu_i^* M_i - C) \geq n-r.$$
 Finally, the complementary slackness relation indicates that the columns of $X^*$ belong to the nullspace of $(\sum_i \mu_i^* M_i-C)$, which is a vector space of dimension at most $n-(n-r)=r$,
 and so we conclude that $\operatorname{rank} X^* \leq r.$\par
\vspace{12pt}

We now turn to the study of the general case in which $M_i\succeq0$.
 To this end, we consider the perturbed problems
\begin{align}
  \max &\quad \langle C,X \rangle + \langle R_0,Y \rangle +\vec{h_0}^T\vec{\lambda} \nonumber \\
\textrm{s.t.} &\quad \langle M_i + \varepsilon I,X \rangle \leq b_i + \langle R_i,Y\rangle
       +\vec{h_i}^T\vec{\lambda} \qquad i \in [l], \label{PCepsilon} \tag{$P^\varepsilon_{\scriptscriptstyle{\textrm{CMB}}}$} \\
&\quad X \succeq 0,\ Y \succeq 0,\nonumber
\end{align}
and
\begin{align}
 \min_{\vec{\mu}\geq0} &\quad \sum_{i=1}^l \mu_i b_i, \nonumber\\
 \textrm{s.t.} &\quad  \sum_{i=1}^l \mu_i (M_i+\varepsilon I) \succeq C, \label{DCepsilon} \tag{$D^\varepsilon_{\scriptscriptstyle{\textrm{CMB}}}$}\\
 &\quad R_0 + \sum_{i=1}^l \mu_i R_i \preceq 0,\nonumber\\
 &\quad \vec{h_0}+H\vec{\mu}=\vec{0}. \nonumber
\end{align}
where $\varepsilon\geq0$. Note that the refined Slater constraints qualification for the unperturbed problems~\eqref{SDPPackingExt}
and~\eqref{SDPPackingDualExt} (i.e.\ simultaneous feasibility (resp.\ strict feasibility) of all the
affine constraints  (resp.\ non-affine constraints)) implies the qualification of the constraints
for~\eqref{PCepsilon} and~\eqref{DCepsilon} on a neighborhood $\varepsilon \in [0,\varepsilon_0]$,
$\varepsilon_0>0$. We denote by $\big((X^\varepsilon,Y^\varepsilon,\vec{\lambda^\varepsilon}),\vec{\mu^\varepsilon}\big)$
a pair of primal and dual
solutions of~\eqref{PCepsilon}--\eqref{DCepsilon}.
If $\varepsilon>0$,  $M_i + \varepsilon I \succ 0$ and it follows
from the previous discussion that $X^\varepsilon$ is of rank at most $r$.
We show below that we can choose the optimal variables
$(X^\varepsilon, Y^\varepsilon,\vec{\lambda^\varepsilon},\vec{\mu^\varepsilon})_{\varepsilon \in ]0,\varepsilon_0]}$
within a bounded region, so that we can construct a converging
subsequence $(X^{\varepsilon_k}, Y^{\varepsilon_k}, \vec{\lambda^{\varepsilon_k}},\vec{\mu^{\varepsilon_k}})_{k \in \mathbb{N}},\ \varepsilon_k\to0$
from these variables.
To conclude, we will see that the limit $(X^0,Y^0,\vec{\lambda^0},\vec{\mu^0})$ satisfies the
KKT conditions for Problems~\eqref{SDPPackingExt}--\eqref{SDPPackingDualExt},
and that $X^{0}$ is of rank at most~$r$.

Let us denote the optimal value of Problems~\eqref{PCepsilon}--\eqref{DCepsilon} by $OPT(\varepsilon).$ Since the constraints
of the primal problem becomes tighter when $\varepsilon$ grows, it is clear that $OPT(\varepsilon)$ is nonincreasing with respect to $\varepsilon$,
so that $$\forall \varepsilon \in [0,\varepsilon_0],\  OPT(\varepsilon_0)\leq OPT(\varepsilon)\leq OPT(0).$$
Now let $\varepsilon\in ]0,\varepsilon_0]$. By assumption, there exists a vector $\vec{\mub}\geq\vec{0}$ such that
\begin{equation}
 \sum_i \mub_i (M_i+\varepsilon I) \succeq \sum_i \mub_i M_i \succ C,\quad \textrm{and}\quad R_0+\sum_i \mub_i R_0 \prec 0. \label{sfdual}
\end{equation}
Therefore, we have
\begin{align*}
 OPT(\varepsilon) =\langle C, X^\varepsilon \rangle +  \langle R_0, Y^\varepsilon \rangle + \vec{h_0}^T \vec{\lambda^\varepsilon}
&\leq \big\langle \sum_i \mub_i (M_i+\varepsilon I), X^\varepsilon \big\rangle + \langle R_0, Y^\varepsilon \rangle
  + \vec{h_0}^T \vec{\lambda^\varepsilon}\\
&\leq \sum_i \mub_i \big(b_i+\langle R_i, Y^\varepsilon \rangle + \vec{h_i}^T \vec{\lambda^\varepsilon}\big) + \langle R_0, Y^\varepsilon \rangle
+ \vec{h_0}^T \vec{\lambda^\varepsilon}\\
&= \vec{\mub}^T \vec{b} + \langle \sum_i \mub_i R_i +R_0, Y^\varepsilon \rangle 
+ ( \underbrace{\vec{h_0}+H \vec{\mub}}_{=\vec{0}})^T\vec{\lambda^\varepsilon},
\end{align*}
where the first inequality follows from~\eqref{sfdual}, and the second one from the feasibility condition 
$\langle M_i+\varepsilon I , X^\varepsilon \rangle \leq b_i+\langle R_i, Y^\varepsilon \rangle + \vec{h_i}^T \vec{\lambda^\varepsilon}$.
The assumption~\eqref{sfdual}
moreover implies that $-(\sum_i \mub_i R_i +R_0)$ is positive definite, so that its smallest eigenvalue $\lambda'$ is positive, and
$$\lambda'\ \trace\  Y^\varepsilon \leq \big\langle -(\sum_i \mub_i R_i +R_0), Y^\varepsilon \big\rangle
       \leq  \vec{\mub}^T \vec{b}- OPT(\varepsilon)\leq  \vec{\mub}^T \vec{b}- OPT(\varepsilon_0).$$
This shows that the trace of $Y^\varepsilon$ is bounded, and so $Y^\varepsilon\succeq0$ is bounded. 

Similarly, to bound $X^\varepsilon$, we write:
\begin{align*}
\langle \sum_i \mub_i M_i -C, X^\varepsilon \big\rangle
& \leq
\langle \sum_i \mub_i (M_i+\varepsilon I) -C, X^\varepsilon \big\rangle\\
 &=\langle \sum_i \mub_i (M_i+\varepsilon I) , X^\varepsilon \big\rangle   
-OPT(\varepsilon)+\langle R_0, Y^\varepsilon \rangle +\vec{h_0}^T \vec{\lambda^\varepsilon} \\
& \leq  \sum_i \mub_i \big(b_i+\langle R_i, Y^\varepsilon \rangle + \vec{h_i}^T \vec{\lambda^\varepsilon}\big) 
-OPT(\varepsilon)+\langle R_0, Y^\varepsilon \rangle +\vec{h_0}^T \vec{\lambda^\varepsilon} \\
& = \vec{\mub}^T \vec{b} -OPT(\varepsilon) +\underbrace{\langle \sum_i \mub_i R_i +R_0, Y^\varepsilon \rangle}_{\leq 0} +  
 ( \underbrace{\vec{h_0}+H \vec{\mub}}_{=\vec{0}})^T\vec{\lambda^\varepsilon},
\end{align*}
where the first equality comes from the expression of $OPT(\varepsilon)$. The matrix $\sum_i\mub_i M_i -C$ is positive definite
and its smallest eigenvalue $\lambda''$ is therefore positive. Hence,
$$\lambda''\ \trace\ X^\varepsilon \leq  \vec{\mub}^T \vec{b} -OPT(\varepsilon) \leq  \vec{\mub}^T \vec{b}- OPT(\varepsilon_0),$$
and this shows that the matrix $X^\varepsilon\succeq 0$ is bounded.

Now, note that the feasibility of $\vec{\lambda^\varepsilon}$ implies that the quantity
$b_i + \langle R_i,Y^\varepsilon \rangle +\vec{h_i}^T\vec{\lambda^\varepsilon}$ is nonnegative for all $i\in[l]$.
Since $Y^\varepsilon$ is bounded, we deduce the existence of a lower bound $m_i\in\R$ such that
$\vec{h_i}^T\vec{\lambda^\varepsilon}\geq m_i$ ($\forall i\in [l]$). Similarly, since
$\vec{h_0}^T\vec{\lambda^\varepsilon}\geq OPT(\varepsilon_0) - \langle C, X^\varepsilon \rangle - \langle R_0, Y^\varepsilon \rangle$,
there is a scalar $m_0$ such that $\vec{h_0}^T\vec{\lambda^\varepsilon}\geq m_0$.
We now use the fact that every vector ($-\vec{h_i}$) may be written as a positive combination of the $\vec{h_k},$ $(k\in\{0\}\cup[l])$,
and we obtain that the quantities $\vec{h_i}^T\vec{\lambda^\varepsilon}$ are also bounded from above.
Let us denote by $H_0$ the matrix $[\vec{h_0},H]$; we have just proved that the vector $H_0^T \vec{\lambda^\varepsilon}$ is bounded:
$$\exists \overline{m}\in\R:\ \quad \Vert H_0^T \vec{\lambda^\varepsilon} \Vert_2 \leq \overline{m}$$
(the latter bound does not depend on $\varepsilon$). Note that one may assume without loss of generality that
$\vec{\lambda^\varepsilon}\in \range H_0$ (otherwise we consider the projection $\vec{\lambda_P^\varepsilon}$ of
$\vec{\lambda^\varepsilon} \ \textrm{on}\ \range H_0$ which is also a solution since $H_0^T\vec{\lambda^\varepsilon}=H_0^T\vec{\lambda_P^\varepsilon}.$
We know from the Courant-Fisher theorem that
the smallest positive eigenvalue of $H_0 H_0^T$ satisfies:
$$\lambdamin^>(H_0 H_0^T)=\min_{\vec{v}\in \range H_0 \setminus \{\vec{0}\}}\
 \frac{\vec{v}^T H_0 H_0^T \vec{v}}{\vec{v}^T \vec{v}}.$$
Therefore, since we have assumed $\vec{\lambda^\varepsilon}\in \range H_0$:
$$\Vert\vec{\lambda^\varepsilon}\Vert^2 \leq \frac{\Vert H_0^T \vec{\lambda^\varepsilon}\Vert^2}{\lambdamin^>(H_0 H_0^T)}
\leq \frac{\overline{m}^2}{\lambdamin^>(H_0 H_0^T)}.$$

It remains to show that the dual optimal variable $\vec{\mu^\varepsilon}$ is bounded. Our strict primal feasibility
assumption (which does not entail generality thanks to Lemma~\ref{lem:bipositiflemma})
ensures the existence of a matrix $\Yb\succ0$ and a vector $\vec{\lb}$ such that
$$\forall i\in[l],\ \langle R_i, \Yb \rangle +b_i +\vec{h_i}^T\vec{\lb} =\eta_i >0.$$
By dual feasibility, $R_0+\sum_i \mu_i^\varepsilon R_i$ is a negative semidefinite matrix, and we have:
$$0\geq \langle R_0, \Yb \rangle +\sum_{i=1}^l \mu_i^\varepsilon \langle R_i, \Yb \rangle
=\langle R_0, \Yb \rangle +\sum_{i=1}^l \mu_i^\varepsilon (\eta_i - b_i-\vec{h_i}^T\vec{\lb}).$$
Hence, we have the following inequalities:
\begin{align*}
\forall k \in[l],\ \eta_k  \mu_k^\varepsilon \leq \sum_{i=1}^l  \eta_i  \mu_i^\varepsilon &\leq
  \vec{b}^T\vec{\mu^\varepsilon} + \vec{\lb}^T H \vec{\mu^\varepsilon} - \langle R_0, \Yb \rangle\\
&= OPT(\varepsilon) -  \vec{\lb}^T \vec{h_0} - \langle R_0, \Yb \rangle\\
& \leq OPT(0)-  \vec{\lb}^T \vec{h_0} - \langle R_0, \Yb \rangle,
\end{align*}
and we have shown that $\vec{\mu^\varepsilon}\geq \vec{0}$ is bounded.

\sloppypar{~\\ \indent We can therefore construct a
sequence of pairs of primal and dual optimal solutions
$(X^{\varepsilon_k}, Y^{\varepsilon_k},\vec{\lambda^{\varepsilon_k}}, \vec{\mu^{\varepsilon_k}})_{k \in \mathbb{N}}$
that converges, with $\varepsilon_k \underset{k \to \infty}{\longrightarrow} 0$, $\varepsilon_k>0$.}
In this sequence, the limit $X^0$ of $X^{\varepsilon_k}$ is of rank at most $r$, because the rank is a lower semicontinuous function and $\rank\ X^{\varepsilon_k}\leq r$
for all $k\in\mathbb{N}$.
It remains to show that $(X^{0}, Y^{0}, \vec{\lambda^0})$ is a solution of Problem~\eqref{SDPPackingExt}. The $\varepsilon-$perturbed KKT conditions must
hold for all $k\in\mathbb{N}$, and so they hold for the pair $\big((X^{0}, Y^{0}, \vec{\lambda^0}), \vec{\mu^{0}}\big)$
by taking the limit (this works because $\mathbb{S}_n^+$ is closed).
This concludes the proof of the existence of a solution in which $\rank\ X\leq r$.
\vspace{12pt}

It remains to show the second statement of this theorem, namely that if
$C\neq 0$ and $\overline{r}:=\displaystyle{\min_{i\in[l]}}\ \rank\ M_i$, then
the rank of $X$ is bounded by $n-\overline{r}+r$ for any solution $(X,Y,\vec{\lambda})$ of~\eqref{SDPPackingExt}.

Let $(X^*,Y^*,\vec{\lambda^*})$ be a solution of Problem~\eqref{SDPPackingExt}. 
If the primal problem is strictly feasible, then there exists a Lagrange multiplier $\vec{\mu^*}\geq\vec{0}$
such that the KKT conditions described at the beginning of this proof are satisfied. Since
$C\neq0$, we have $\vec{\mu^*}\neq\vec{0}$, and we can write:
$$\rank\ (\sum_{i\in[l]} \mu_i^* M_i -C)\geq \overline{r}-r.$$
Hence, since by complementary slackness, $X^*$ belongs to the nullspace of
$(\sum_{i\in[l]} \mu_i^* M_i -C)$, we find $\rank\ X^*\leq n-\overline{r}+r$.

If the primal problem is not strictly feasible, there must be an index $i\in[l]$ such
that $\langle M_i, X^*\rangle=0$ (otherwise, $(\eta_1I,Y^*+\eta_2 I,\vec{\lambda^*})$ would be
strictly feasible for sufficiently small positive reals $\eta_1$ and $\eta_2$). Therefore,
$X^*$ is in the nullspace of a matrix of rank larger than $\overline{r}$, and
$\rank\ X^*\leq n-\overline{r}\leq n-\overline{r}+r$.
\end{proof}

\section{Proof of Theorem~\ref{theo:SDPSol1Asy}}

We assume that Problems~\eqref{SDPPackingExt} and~\eqref{SDPPackingDualExt}
are feasible, and for $\eta \geq 0$ we consider the following pair of primal and dual perturbed problems.
\begin{align}
  \sup &\quad \langle C,X \rangle + \langle R_0,Y \rangle +\vec{h_0}^T\vec{\lambda} \nonumber \\
\textrm{s.t.} &\quad \langle M_i,X \rangle \leq b_i + \langle R_i,Y\rangle
       +\vec{h_i}^T\vec{\lambda} \qquad i \in [l], \label{Peta} \tag{$P_\eta$} \\
&\quad \eta\ (\trace\ X+\trace\ Y)\leq 1, \nonumber \\
&\quad X \succeq 0,\ Y \succeq 0,\nonumber
\end{align}
and
\begin{align}
 \inf_{\vec{\mu}\geq\vec{0},\ \sigma\geq 0} &\quad \sum_{i=1}^l \mu_i b_i+\sigma, \nonumber\\
 \textrm{s.t.} &\quad  \sum_{i=1}^l \mu_i  M_i + \sigma \eta I \succeq C, \label{Deta} \tag{$D_\eta$}\\
 &\quad R_0 + \sum_{i=1}^l \mu_i R_i -\sigma \eta I \preceq 0,\nonumber\\
 &\quad \vec{h_0}+H\vec{\mu}=\vec{0}. \nonumber
\end{align}
It is clear that the feasibility of Problem~\eqref{SDPPackingExt} implies that of~\eqref{Peta} if
$\eta>0$ is sufficiently small. Let $\vec{\mub}$ be a dual feasible variable for
Problem~\eqref{SDPPackingDualExt}, and $\sigma>0$ be sufficiently large so that
$\sum_{i=1}^l \mu_i  M_i + \sigma \eta I \succ C$ and
$R_0 + \sum_{i=1}^l \mu_i R_i -\sigma \eta I \prec 0$: the refined Slater condition holds for the
perturbed problem~\eqref{Deta}. Hence, by Theorem~\ref{theo:SDPSol1Ext}, there exists a solution
$(X^\eta,Y^\eta,\vec{\lambda^\eta})$ of Problem~\eqref{Peta} in which $\rank\ X^\eta\leq r$. We will
show that $\langle C,X^\eta \rangle + \langle R_0,Y^\eta \rangle +\vec{h_0}^T\vec{\lambda^\eta}$
converges to the value of the supremum in Problem~\eqref{SDPPackingExt} as $\eta\to0^+$,
which will complete this proof.

Let $\eta_k$ be a positive sequence decreasing to $0$, and define
$\gamma_k:=\langle C,X^{\eta_k} \rangle + \langle R_0,Y^{\eta_k} \rangle +\vec{h_0}^T\vec{\lambda^{\eta_k}}$.
It is clear that $\gamma_k$ is a nondecreasing sequence, because the constraints in Problem~\eqref{Peta}
become looser as $\eta$ gets smaller, and $\gamma_k$ is bounded from above by the value of the supremum
$\gamma^*$ in Problem~\eqref{SDPPackingExt}. Therefore, $(\gamma_k)_{k \in \mathbb{N}}$ converges. 
Assume (\emph{ad absurdum}) that the limit of this sequence is $\gamma_\infty<\gamma*$. Then, there are some variables
$(X_0,Y_0,\vec{\lambda_0})$ that are feasible for~\eqref{SDPPackingExt}, and such that 
$\langle C,X_0 \rangle + \langle R_0,Y_0 \rangle +\vec{h_0}^T\vec{\lambda_0}>\gamma_\infty$.
But then, $(X_0,Y_0,\vec{\lambda_0})$ is also feasible for Problem~\eqref{Peta}, when
$\eta\leq \eta_0:=(\trace\ X_0+ \trace\ Y_0)^{-1}$. For any $k\in\mathbb{N}$ such that $\eta_k\leq \eta_0$,
this contradicts the optimality of $(X^{\eta_k},Y^{\eta_k},\vec{\lambda^{\eta_k}})$ for Problem~($P_{\eta_k}$).
Hence, $\gamma_\infty=\gamma*$ and the proof is complete.
\end{document}